\documentclass[preprint,12pt,authoryear]{elsarticle}

\journal{Linear Algebra and its Applications}

\usepackage{amssymb,amsmath}
\setcitestyle{square,sort,comma,numbers}

\usepackage{graphicx}

\usepackage{lineno,hyperref}
\usepackage{amssymb}
\usepackage[top=1in, bottom=1in, left=1in, right=1in]{geometry}
\usepackage{amssymb,amsfonts}
\usepackage{bbm}
\usepackage[all,arc]{xy}
\usepackage{enumerate}
\usepackage{mathptmx}
\usepackage{mathrsfs}
\usepackage[T1]{fontenc}
\usepackage[applemac]{inputenc}
\usepackage[english]{babel}
\usepackage{times}
\usepackage{multirow}
\usepackage{color}
 \usepackage[english]{babel}
\usepackage{graphicx}
\usepackage{float}
\usepackage{xcolor}
 \usepackage{amsmath}
\usepackage{graphicx}
\usepackage{multicol}
\usepackage{amsmath}
\usepackage{subeqnarray}
\usepackage{url}
\usepackage[bottom]{footmisc}

 \usepackage{amsthm}
\usepackage{relsize}
\usepackage{bm}
\usepackage{yfonts}
\usepackage[mathcal]{euscript}
\usepackage{url}
\usepackage{etex}

\usepackage{braket}
\usepackage{tikz-cd}

\usepackage{verbatim}

\usepackage{graphicx}
\DeclareMathOperator{\ii}{{\bf i}}

\newtheorem{thm}{Theorem}[section]
\newtheorem{cor}[thm]{Corollary}
\newtheorem{prop}[thm]{Proposition}
\newtheorem{lem}[thm]{Lemma}

\newtheorem{application}{Application}

\theoremstyle{definition}

\newtheorem{rem}{Remark}

\modulolinenumbers[5]
\usepackage{psfrag}
\usepackage[inline]{enumitem}   
\makeatletter
\newcommand{\inlineitem}[1][]{%
\ifnum\enit@type=\tw@
    {\descriptionlabel{#1}}
  \hspace{\labelsep}%
\else
  \ifnum\enit@type=\z@
       \refstepcounter{\@listctr}\fi
    \quad\@itemlabel\hspace{\labelsep}%
\fi}
\makeatother
\begin{document}


\begin{frontmatter}

\title{Characteristic adjacency matrix associated with a hypergraph.}

\author[A]{Alain Bretto}
\ead{alain.bretto@unicaen.fr}
\author[B]{Alain Faisant}
\ead{faisant@univ-st-etienne.fr}
\address[A]{Normandie Univ-Caen, GREYC CNRS UMR-6072, Campus II, Bd Marechal Juin BP 5186, 14032 Caen cedex 5, France}
\address[B]{Univ Lyon, UJM-Saint-\'Etienne, CNRS, ICJ UMR 5208, 42023 Saint-\'Etienne, France}






\begin{abstract}
We define a new matrix associated with a hypergraph. A study of this matrix is carried out, in particular the study of its spectrum, which shows the relevance of the construction.
 We demonstrate that this adjacency matrix is characterisitic of the hypergraph, that is to say, two isomorphic hypergraphs have similar matrices. Furthermore, starting from this matrix, we can reconstruct the hypergraph. Starting from this matrix, we introduce new graphs associated with the hypergraph.
\end{abstract}


 \end{frontmatter}

\section{Introduction}\label{section1}
Algebraic graph theory is a field that has been widely studied for several decades,  \cite{Bigg1974,Godsil2001}. This method yields elegant and particularly relevant results regarding the properties of the 
graph being studied. An important part of this study is the spectral theory of graphs,  \cite{Bapat2010,bookBretto2022}. This involves studying the adjacency matrix and the matrices constructed from it, such as, for example, the Laplacian.
The value of studying the adjacency matrix lies in the fact that it fully characterizes the graph and therefore provides information that allows for an intrinsic analysis of that graph. \\

Algebraic hypergraph theory is much more recent \cite{Shetty2025}; this approach to hypergraphs introduces several problems that do not exist in graph theory. The main obstacle is that the adjacency matrix of a hypergraph is not characteristic of it, \cite{bookHyperBretto2026}. However, most of the matrices studied in the context of graph theory come from the adjacency matrix and this is also the case in the context of hypergraphs. Therefore, the information conveyed by these matrices is not as precise as in the case of graphs. To address this problem, several authors have introduced the concept of tensors or hypermatrices. \cite{Cooper2012,Hou2019,Liu2016,Cardoso2019,Shetty2025}. 
These algebraic tools yield interesting results, and the eigenvalues provide valuable information that sheds light on the combinatorial properties of hypergraphs.\\

While these representations are relevant, they raise other problems that are very difficult to solve. One of the major drawbacks of studying tensor (hypermatrix) hypergraph theory is the complexity class of the eigenvalue computation, which is NP-hard, \cite{BANERJEE2021,Hillar2013}. Moreover, most tensor properties belong to this complexity class (NP-hard), \cite{Hillar2013}. Consequently, particularly for applications,\cite{Kurian2024}, it is quite difficult to use these mathematical results in the design of efficient algorithms.\\

In this article we introduce a new definition of hypergraphs, which will allow us to define a new adjacency matrix whose rows (resp. columns) are indexed by couples :(vertex, hyperedge). Clearly, this type of matrix allows for the unique reconstruction of the associated hypergraph; moreover, the Theorem \ref{characteristionHypergraph}  shows that this matrix representation fully characterizes the hypergraph. A study of the associated spectrum is conducted, thus demonstrating the relevance of the information conveyed by this matrix. Using this matrix, we introduce new graphs associated with the hypergraph and show that they are strongly linked to the combinatorial and topological nature of the hypergraph under study.

\section{definitions}
We remind the reader that a  \emph{hypergraph} $H$ on a set $V$  is a triplet  $H = (V; E, \varepsilon)$  where

\begin{description}
\item[-] $V$ is a set called \index{vertex} \emph{set of vertices};
\item[-] $E$ is a set  called \index{hyperedge} \emph{set of hyperedges};
\item[-] $\varepsilon: E\longrightarrow P(V)=2^V$, ($ P(V)$ is the set of subsets -power set- of $V$) is a function called \index{incidence function} \emph{incidence function} which associates to any hyperedge $e\in E$ the set of vertices of $e$.
\end{description}

We can give, now  another definition  of a \emph{hypergraph}: Let $H=(V;E,\varepsilon)$ be a hypergraph; we can define $H$ as follows:
\[
 H:=V\ast E= \{(x, e)\in V\times E: x\in \varepsilon(e)\};
 \]
(sometimes $x \ast e$ is written for $(x, e) \in V \ast E$).This set provides all the information about the hypergraph, in particular:
\begin{itemize}
\item[-] $\forall e \in E: e$ is empty if $\forall x \in V$ we have $(x,e) \not \in V \ast E$;
\item[-] $\forall x \in V: x$ is isolated if $\forall e  \in E$ we have $(x,e) \not \in V \ast E$.
\end{itemize}
Conversly let $A \subseteq V \times E$: define $H=(V;E, \varepsilon)$ with for $e\in E$,  $\varepsilon(e):=\{x \in V:\\
(x,e)\in A\}: H$ is a hypergraph such that $V \ast E =A$.
In this section we suppose $H=(V; E, \varepsilon)$ without isolated vertex, and $\varepsilon(e) \neq \emptyset$ so that all vertices and hyperedges are present in the set $V *E$.

The   \emph{star} $H(x)$ centered in $x$ is the set  of hyperedges $e \in E$ "containing" $x$, 
i.e.
$$ H(x)=\{e \in E : x \in \varepsilon(e) \}$$

\noindent  The parameter $d(x) = \vert H(x)\vert$  is the  \emph{degree} of $x$. The maximal degree of a hypergraph $H$ is denoted by $\Delta (H)$ and the minimal degree is denoted by $\delta (H)$. \\
We will  note by ${\displaystyle d(e)=\sum_{x\in e}(d(x)-1)}$ the \index{degree of a hyperedge} \emph{degree of a hyperedge}.\\
We will define the \emph{dual} $H^{*}= (V^*; E^*, \varepsilon^{*})$, ($V^*:=E$ and $E^*:=V$), by
\noindent \[
V^*\ast E^*:= \{(X_{e}, E_{x}):X_{e}\in \varepsilon^{*}(E_{x}), \; \textrm{iff}\; \varepsilon(e)\ni x\}.
\]

\section{ The adjacency  matrix $A_H$ }

\begin{rem}
For convenience, in this section  we denote $VE = \{xe, x \in \varepsilon(e) \}$ instead of $V*E$ and $x*e$ respectively.

\noindent We consider $VE$ as an orthonormal basis of the Hermitian  $\mathbb{C}$-vector space: $<VE>$ of dimension $m_E = \sum_{e \in E}\vert \varepsilon(e)\vert$.

\noindent Therefore, an element $\xi \in <VE>$ is written as:
\[
\xi = \sum_{xe \in VE} \alpha_{xe}x e,\; \textrm{with}\; \alpha_{xe} \in \mathbb{C}
\]
Let ${\displaystyle \eta = \sum_{xe}\beta_{xe}}xe$ be an element of the Hermitian ${\displaystyle \mathbb{C}}$-vector space ${\displaystyle \braket{VE}}$;

\noindent we have:
\[
\braket{\xi \vert\eta}= \sum_{xe} \alpha_{xe} \overline{\beta}_{xe}, \; \textrm{et}\; ||\xi ||^2=\sum_{xe} \alpha_{xe} \overline{\alpha}_{xe}
\]
\end{rem}

\subsection{ Endomorphism $\mathfrak{A}$ and adjacency matrix $A_H$ associated with a hypergraph}
\subsubsection{Order associated to a Hypergraph}

Let's choose an order on $V$ and $E$: $V=\{x_1, \ldots, x_n\}, E=\{e_1, \ldots , e_m\}$; from these orders we can order $VE$ by reverse lexocographic order ($VE^{rl}$):
 \[
 xe \le x'e' \;\;\textrm{ iff}\;\;\; e < e',\;  (e=e_i, e'=e_j,  i<j),\;  \textrm{or } \;\;\;  e=e'  \; \textrm{and} \;\; x\le x',\; ( x=x_i, x'=x_j,  i \le j).
 \]
We can also order $VE$ lexicographically:
\[
xe \le x'e'\;\; \textrm{iff} \; x<x'\;\; \textrm{or}\; x=x' \; \; \textrm{and}\; e \le e'
\]
 we denote $VE^{l}$ this ordered set.\\
 
\noindent  {\bf Eaemple:}   In the hypergraph of Figure \ref{matriceHypergarph}, the order of the vertices and hyperedges is respectively $V=\{x_1, x_2, x_3, x_4, x_5 \}$ and $E=\{e_1, e_2, e_3 \}$. 
Therefore, the reverse lexicographic order of $VE$ is:

\[ 
x_{1}e_1; x_{2}e_1; x_{3}e_1; x_{1}e_{2}; x_{4}e_{2}; x_{3}e_{3}; x_{5}e_{3}
\]
and the lexicographic order of $VE$ is:
\[ 
x_{1}e_1; x_{1}e_2; x_{2}e_1; x_{3}e_{1}; x_{3}e_{3}; x_{4}e_{2}; x_{5}e_{3}
\]

\subsubsection{Endomorphism associed to a hypergraph}
Define $\mathfrak{A}_H=\mathfrak{A} \in \mathrm{End}( \braket{VE})$ on the basis $VE$ by the rule:

\begin{equation}
\mathfrak{A}(xe)=\sum_{y \neq x} ye +\sum_{a > e} \ii xa  -\sum_{a < e} \ii xa
\end{equation}

\begin{rem}
For convenience and when there is no ambiguity, we will write:
\[
\mathfrak{A}(xe)=\sum_{y \neq x} ye +\sum_{a \neq e} \pm \ii xa
\]
\end{rem}
The coefficients of $\mathfrak{A}(xe)$ are Gaussian integers: elements of $\mathbb{Z}[\ii]$.\\

\noindent  A square matrix can be associated with this endomorphism.
\subsubsection{Endomorphism matrix}
The matrix $A_H=A$ of $\mathfrak{A}$ in the basis $VE$ is $A=\left(a_{  x_ie_j, x_{i'}e_{j'}    }\right) $ with

\[
a_{ xe, x'e'}=\begin{cases}
1 \; \;\;\textrm{if}\;\;\; x\neq x' \wedge e=e';\\
0 \;\;\;\textrm{if} \; \;\; x=x'\wedge e=e';\\
+\ii \; \;\; \textrm{if} \;\;\;x=x'\wedge  e< e';\\
-\ii\;\;\; \textrm{if}\;\;\;x=x'\wedge  e> e';\\
0 \; \;\; \textrm{otherwise}

\end{cases}
\]

We will denote $A_H$  the \emph{adjacency matrix} \index{adjacency matrix}of $H$

 \begin{figure}[!ht]

\psfrag{1}{$x_{1}$}
\psfrag{2}{$x_{2}$}
\psfrag{3}{$x_{3}$}
\psfrag{4}{$x_{4}$}
\psfrag{5}{$x_{5}$}
\psfrag{a}{$e_{1}$}
\psfrag{b}{$e_{2}$}
\psfrag{c}{$e_{3}$}
\centering
\includegraphics[width=5cm]{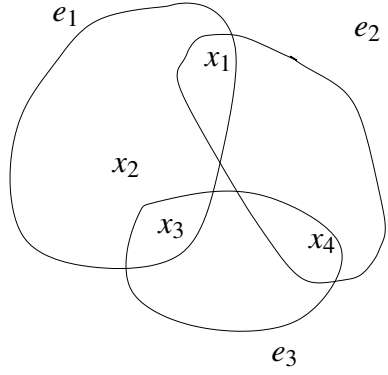}
\caption{The matrix below is the adjacency matrix $A_H$ of the hypergraph $H$ in this figure: $ V=\{x_1, x_2, x_3, x_4, x_5 \}, E=\{e_1, e_2, e_3 \}$. }
\label{matriceHypergarph}
\end{figure}
\bordermatrix{& x_{1}e_1 & x_{2}e_1&x_{3}e_1&x_{1}e_{2}&x_{4}e_{2}&x_{3}e_{3}&x_{4}e_{3}
\cr  x_{1}e_1  & 0 & 1 & 1& +\ii& 0 & 0 & 0   
 \cr x_{2}e_1 & 1 & 0 & 1 & 0 & 0 & 0 & 0 
 \cr x_{3}e_1 & 1 & 1 & 0 & 0 & 0 &+ \ii & 0
 \cr x_{1}e_{2}& -\ii& 0 & 0 & 0 & 1 & 0 & 0
 \cr x_{4}e_{2}& 0 & 0 & 0 & 1 & 0 & 0 & +\ii
 \cr x_{3}e_{3}& 0 & 0 & -\ii & 0 & 0 & 0 & 1
 \cr x_{4}e_{3}& 0 & 0 & 0 & 0 & -\ii & 1 & 0}
\mbox{}\\  

 
 \section{Basic properties of endomorphism $\mathfrak{A}$ and adjacency matrix $A_H$ }
 This section presents the initial results derived from the endomorphism $\mathfrak{A}$ and its associated matrix $A_H$.
  
\begin{prop}\label{resultgeneral}\mbox{}

\begin{itemize}
\item[i)] The endomorphism $\mathfrak{A}_H$ and the matrix $A_H$ are hermitian ( $\overline{A^t}=A$) so are diagonalizable, with real eigenvalues;
\item[ii)] Each $e \in E$ corresponds to a maximum principal " $1$-block" along the diagonal: it is a square Hermitian submatrix of dimension
$\vert \varepsilon(e)\vert \times \vert\varepsilon(e)\vert$ composed of $0$ on the diagonal and $1$ elsewhere;
\item[iii)]   let ${\displaystyle (x, e)}$ be a  row  (resp. the column) of $A(H)$,  the number of $1$ on the row ${\displaystyle (x, e)}$ (resp. on the column ${\displaystyle (x, e)}$) is ${\displaystyle \vert e\vert -1}$
and the number of  $\ii$ on the row ${\displaystyle (x, e)}$ (resp. on the column ${\displaystyle (x, e)}$) is ${\displaystyle d_{H}(x)-1}$;
\item[iv)] Similarly, the matrix $A_H^{l}$ of $\alpha$ in the ordered basis $VE^{l}$ is Hermitian, and to each $x \in V$ corresponds a maximal principal Hermitian submatrix of "$\ii$-block" of 
dimension $d(x) \times d(x)$ composed of $0$ on the diagonal, $+\ii$ above the diagonal, and $-\ii$ below the diagonal.
\end{itemize}
\end{prop}

\begin{proof}
\mbox{}
\begin{itemize}
\item[-] Indeed, because $a_{x'e',xe}= \overline{a}_{xe,x'e'}$;
\item[-]  assertion iii) comes from the construction of the matrix;
\item[-] Assertions ii) and iv) come from the reverse lexicographic order and the lexicographic order on $VE$, respectively.
\end{itemize}

\end{proof}

\begin{rem}\label{remSimilar}
Changing the order of V and E isn't really important;  indeed:

\begin{itemize}
\item[-] the change of order on $V=\{ x_{i_1}, \ldots , x_{i_k}, \ldots x_{i_n} \}$ corresponds to a bijection $\sigma (i_{k}) \in \{1, \ldots , n \}, 1 \le k \le n$;
\item[-] the change of order on $E=\{e_{j_1} \ldots , e_{j_l}, \ldots e_{j_m} \}$ corresponds to a bijection $\tau(j_{l}) \in \{1, \ldots, m\}, 1 \le l \le m$;
\item[-] The function $f : H \longrightarrow H$ defined by $f(x_{i_k})= x_{\sigma(i_k)}$ and $f(e_{j_l})=e_{\tau(j_{l})}$ is an automorphism of $H$ and gives rise t
o $\alpha'= f \circ \alpha \circ  f^{-1}$, which is a conjugate of $\alpha$ and the matrix $A'_{H}$ of $\alpha'$ is similar to the matrix $A_{H}$ of $\alpha$.
\end{itemize}
\end{rem}
Matrix below is the matrix $A_H^{l}$ from the figure \ref{matriceHypergarph}, this one  illustrates  Remark \ref{remSimilar}

\bordermatrix{& x_{1}e_1 & x_{1}e_2&x_{2}e_1&x_{3}e_{1}&x_{3}e_{3}&x_{4}e_{2}&x_{4}e_{3}
\cr  x_{1}e_1  & 0 & +\ii& 1& 1& 0 & 0 & 0   
 \cr x_{1}e_2 & -\ii & 0 & 0 & 0 & 0 & 1 & 0 
 \cr x_{2}e_1 & 1 & 0 & 0 & 1 & 0 &0 & 0
 \cr x_{3}e_{1}& 1 & 0 & 1 & 0 & +\ii& 0 & 0
 \cr x_{3}e_{3}& 0 & 0 & 0 & -\ii& 0 & 0 & 1
 \cr x_{4}e_{2}& 0 & 1 & 0 & 0 & 0 & 0 & +\ii
 \cr x_{4}e_{3}& 0 & 0 & 0 & 0 & 1 & -\ii & 0}
\mbox{}\\\\  

\noindent We can see in this example that the $\ii$-blocks correspond to the vertices $x_1, x_2, x_3, x_4, x_5$ and the dimension of the block $x_i$ is $d(x_i)$.

  
\begin{prop}\label{constructionMatriceChap7}\mbox{}\\
Let $A_H$ be the matrix of $\mathfrak{A}_H$ in the basis $VE$; then the matrix of $\mathfrak{A}_{H*}$ in the basis $V^*E^*$ is obtained from $\mathfrak{A}_H$ by applying the following rules:

\begin{itemize}
\item replace $\ii$ and $-\ii$ with $1$;
\item replace each $1$ below the diagonal with a $-\ii$;
\item replace each $1$ above the diagonal with a $+\ii$.
\end{itemize}
\end{prop}
\begin{proof}
\mbox{}
Indeed, since  $V^*:=E$ and $E^*:=V$, by definition.
\end{proof}
The dual of Figure \ref{matriceHypergarph} is shown in the figure below:

\begin{figure}[!ht]
\psfrag{b}{$X_{e_{1}}$}
\psfrag{c}{$X_{e_{2}}$}
\psfrag{d}{$X_{e_{3}}$}
\psfrag{A}{$E_{x_{1}}$}
\psfrag{B}{$E_{x_{4}}$}
\psfrag{C}{$E_{x_{3}}$}
\psfrag{D}{$E_{x_{2}}$}
\psfrag{F}{$6$}
\psfrag{G}{$7$}
\psfrag{H}{$8$}
\psfrag{K}{$9$}
\psfrag{L}{$10$}
\centering
\includegraphics[width=9cm]{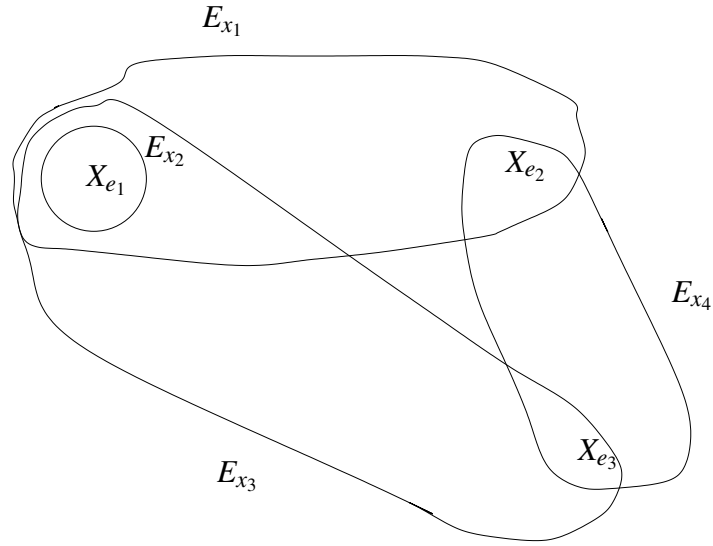}
\caption{This figure shows the dual of the hypergraph of Figure \ref{matriceHypergarph}, the set of vertices is
$V^*:=\{X_{e_{1}}, X_{e_{2}}, X_{e_{3}} \}$ and the set of hyperedges is $E^*:=\{E_{x_{1}}, E_{x_{2}}, E_{x_{3}}, E_{x_{4}}, E_{x_{5}}\}$.}
\label{dualChapitre7}
\end{figure}

The matrix below is that of the hypergraph shown in Figure \ref{dualChapitre7}, it illustrates Proposition \ref{constructionMatriceChap7}.

and the matrix of $\mathfrak{A}_{H^*}$ in basis $V^*E^*=EV$ is \\

\bordermatrix{&e_1x_1 &e_2x_1&e_1x_2&e_1x_3&e_3x_3&e_2x_4&e_3x_4
\cr  e_1x_1  & 0 & 1& +\ii& +\ii& 0 & 0 & 0   
 \cr e_2x_1 & 1& 0 & 0 & 0 & 0 & +\ii& 0 
 \cr e_1x_2 & -\ii & 0 & 0 & +\ii & 0 &0 & 0
 \cr e_1x_3& -\ii & 0 & -\ii & 0 & 1 & 0 & 0
 \cr e_3x_3 & 0 & 0 & 0 & 1& 0 & 0 & +\ii
 \cr e_2x_4 & 0 & -\ii & 0 & 0 & 0 & 0 & 0
 \cr e_3x_4& 0 & 0 & 0 & 0 & -\ii& 0 & 0}
\mbox{}\\\\

\noindent We can check here the previous rule and see the $2$, $1$-blocks of dimension $2\times 2 $ and the $3$, $1$-blocks of dimension $1\times 1$ 
corresponding to the three loops of $H^*$.

 \begin{rem}
 We remind the reader that the numbers $\frac{<\alpha(\xi) | \xi>}{<\xi | \xi>}=R(\mathfrak{A}, \xi)$ are the so-called   \index{Rayleigh quotients} \emph{Rayleigh quotients}. 
 \end{rem}

\begin{prop}\label{firstProperties}

Let $H=(V; E,\varepsilon)$ be a hypergraph, we have the following properties:
\begin{itemize}
\item[i)] $\mathfrak{A}_H \in \mathrm{End}( \braket{VE})$ is diagonalizable , $\mathrm{Spec}( \mathfrak{A}) \subseteq \overline{\mathbb{Q}} \cap \mathbb{R}$ and the proper subspaces $H_{\lambda}$ are orthogonal;
\item[ii)] For all $||\xi||=1$, we have:
\[
 \lambda_{min} \le \braket{\mathfrak{A}(\xi)|\xi} \le \lambda_{max}
\]
with: ${\displaystyle \lambda_{min}= inf_{||\xi||=1}  \braket{\mathfrak{A}(\xi)|\xi}}$ and ${\displaystyle \lambda_{max}= sup_{||\xi||=1}  \braket{\mathfrak{A}(\xi)|\xi} }$;
\item[iii)] If $H=(V; E,\varepsilon)$ has  $e \in E$ such that $|\varepsilon(e)| \ge 2$, then: 
\[
\lambda_{min} \le -1 < +1 \le \lambda_{max};
\]
\item[iv)] We have the following inequalities:
\[
r(H)-1 \le \lambda_{max} \le max_{xe \in VE} \{ d(x)+|\varepsilon(e)|-2 \} \le r(H)-1 + \Delta(H)-1;
\]
\item[v)] ${\displaystyle \sum_{1 \le i \le m_E} \lambda_i =0}$;
\item[vi)] ${\displaystyle \sum_{1 \le i \le m_E} \lambda_i^2 = \sum_{xe\in VE} d(x) + \vert \varepsilon(e)\vert -2}$;
\item[vii)] ${\displaystyle \sum_{xe\in VE} d(x)+ \vert\varepsilon(e)\vert }$ is even.
\end{itemize}
\end{prop}

\begin{proof}  \mbox{}
\begin{itemize}
\item[i)] These are classic properties of Hermitian endomorphisms;
\item[ii)] Let $\{ \pi_1, \ldots , \pi_{m_E} \}$ be an orthonormal basis of eigenvectors associated to the $\lambda_i$, $1 \le i \le m_E$, and let $\xi$ be an element such that $||\xi||=1$, then
\[
\begin{split}
\xi&= \sum_{1 \le i \le m_E} u_i \pi_i \braket{\mathfrak{A}{(\xi)|\xi}}\\
& = \braket{\sum u_i \lambda_i \pi_i | \sum u_i \pi_i}= \sum |u_i|^2 \lambda_i 
\le \lambda_{max}; 
\end{split}
\]
and  ${\displaystyle \braket{\mathfrak{A}(\xi)| \xi}\ge \lambda_{min} \sum |u_i|^2= \lambda_{min}}$. \\
The bounds are obtained using the associated eigenvectors.

\item[iii)]  Let's choose $x \neq x'$ in $\varepsilon(e)$ and consider the following vector: ${\displaystyle \xi=\frac{xe}{\sqrt{2}}+\frac{x'e}{\sqrt{2}}}$, we have:
\begin{itemize}
\item[-] $|| \xi ||=1$;
\item[-] \[
\begin{split}
\braket{\mathfrak{A}{(\xi)|\xi}}&= \frac{1}{2}\braket{\sum_{y \neq x}ye + \sum_{a \neq e} \pm \ii xa + \sum_{y \neq x'}ye + \sum_{a \neq e} \pm \ii x'a | xe+x'e}\\
& =\frac{1}{2}\left[\braket{xe|xe}+ \braket{x'e |x'e}\right]=1\le \lambda_{\max}, \textrm{using  ii)}.
\end{split}
\]
\end{itemize}
Now: let $\eta=\frac{xe}{\sqrt{2}}-\frac{x'e}{\sqrt{2}}$; as above, we have:
\begin{itemize}
\item[-] $ || \eta||=1$;
\item[-] \[
\begin{split} 
\braket{\mathfrak{A}(\eta) | \eta }&=
\frac{1}{2}\braket{ \sum_{y \neq x}ye + \sum_{a \neq e} \pm \ii xa - \sum_{y \neq x'}ye - \sum_{a \neq e} \pm \ii x'a | xe+x'e} \\
&=\frac{1}{2} \left[- \braket{xe|xe}- \braket{x'e |x'e}\right]=-1 \ge \lambda_{\min},  \textrm{using  ii)}.
\end{split}
\]

\begin{rem}
In the case where $|\varepsilon(e)|=1$ for all $e$, we have $<\mathfrak{A}(\xi) |\xi >=0$ for all $\xi$.
\end{rem}
\end{itemize}
\item[iv)] The upper bound is obtained as follows: consider ${\displaystyle \xi=\sum u_{xe} xe}$ an eigenvector associated with ${\displaystyle \lambda_{max}}$: ${\displaystyle \mathfrak{A}(\xi)= \lambda_{max} \xi}$, we now  choose ${\displaystyle x_0e_0}$ such that ${\displaystyle |u_{x_0e_0}|= max |u_{xe}|}$; now let's calculate the component:\\

${\displaystyle \mathfrak{A}(\xi)_{x_0e_0}= \sum u_{xe} (\sum_{y \neq x}+ \sum_{a \neq e} \pm \ii xa)_{x_0e_0}= \lambda_{\max} u_{x_0e_0}}$; so
\[
\begin{split} 
\lambda_{max} |u_{x_0e_0}| &\le  |u_{x_0e_0}| (\sum_{y \neq x} 1 + \sum_{a \neq e } \vert \pm \ii\vert ) \le |u_{x_0e_0}| \left(d(x_0)-1 + |\varepsilon(e_0)|-1\right)\\
& \le |u_{x_0e_0}| \max \{ d(x)+ |\varepsilon(e)|-2;
\end{split} 
\]
 the upper bound follows from the simplification by ${\displaystyle |u_{x_0e_0}|}$.\\

\noindent Let's now move on to the lower bound: \\
let us fix an element $e \in E$ and consider the vector\\
 ${\displaystyle \eta_e=\sum_{x \in \varepsilon(e)}xe}$: ${\displaystyle\braket{\eta_e | \eta_e}= \braket{\sum_x xe |\sum_{x'}x'e }=\sum 1=
|\varepsilon(e)|}$, and\\
 ${\displaystyle \braket{\mathfrak{A}(\eta)|\eta}=\braket{\sum_x \alpha(xe)| \sum_{x'}x'e}=\sum_{x\in \varepsilon(e)}\braket{ \sum_{y \neq x }ye+\sum_{a \neq e} \pm \ii xa|\sum_{x' \in \varepsilon(e)}x'e}}$.\\ 
 Now we distribute the $x'e$: $x'e$ with $\pm \ii xa$, it  gives $0$; $x'e$ with $ye$ only contribute for $x'\neq x$ and $y=x'$; so we found:\\
 ${\displaystyle = \sum_{x \in \varepsilon(e)} \sum_{x' \neq x, x' \in \varepsilon(e)}1=\sum_{x \in \varepsilon(e)}(| \varepsilon(e)|-1)=
|\varepsilon(e)|(|\varepsilon(e)|-1)}$, therefore\\
 ${\displaystyle R(\mathfrak{A}, \eta_e)=\frac{1}{|\varepsilon(e)|} |\varepsilon(e)|(|\varepsilon(e)|-1)= |\varepsilon(e)|-1 \le \lambda_{max}}$ and ${\displaystyle r(H) -1\le \lambda_{max}}$.\\
\item[v)]Indeed, since the diagonal of $A_H$ is zero.
\item[vi)] We have ${\displaystyle a_{xe,x'e'}.a_{x'e',xe}= a_{xe,x'e'} \overline{a}_{xe,x'e'}=1}$ so the diagonal component $(xe,xe)$ of $A_H^2$ is the sum
${\displaystyle Tr A_H^2= \sum_{x'e'} d(x) +|\varepsilon(e)| -2}$.
\item[vii)] Let $\Phi(x)= det (x I -A_H) \in \mathbb{Z}[\ii][x] $ be the characteristic polynomial of $A_H$: the coefficient of $x^{m_E-2} $ is
${\displaystyle c_2=\sum \lambda_i \lambda_j \in \mathbb{Z}[\ii] \cap \mathbb{R}= \mathbb{Z}}$ and\\
 $ {\displaystyle 2c_2=(\sum \lambda_i)^2 - \sum \lambda_i^2= - \sum\lambda_i^2= - \sum d(x)+|\varepsilon(e)| -2}$\\
 and therefore: $ {\displaystyle \sum d(x)+|\varepsilon(e)|= -2c_2-2}$.
\end{itemize}

\end{proof}

\begin{rem} 
If a loop exists, then $\lambda=0 \in \textrm{Spec}(\alpha)$.
\end{rem}
\subsubsection*{Example}
In the hypergraph shown in Figure \ref{matriceHypergarph}, we have:
\begin{description}
\item[-] the number of vertices is $m_E=7$;
\item[-] the characteristic polynomial is:  $P(x)=-x^7+8x^5+2x^{4}-17x^3-6x^2-8x+2$;
\item[-] the spectrum $\textrm{Spec}(\mathfrak{A})$  is:\\
 \[
 \begin{split}
  \lambda_{\max} \simeq 2.382 &> \lambda_2 \simeq 1.601 > \lambda_3 \simeq 0.740 > \lambda_4 \simeq -0.236 > \lambda_5 \simeq  \\
  & -1.000 >\lambda_6 \simeq -1.545 > \lambda_{\min} \simeq -1.942.
 \end{split}
 \]
\end{description}
We can easily verify that ${\displaystyle \lambda_{max} \le max\{ d(x)+|\varepsilon(e)| -2 \}=3}$ and that\\
 ${\displaystyle \frac{1}{m_E} \sum_{e \in E} |\varepsilon(e)|(|\varepsilon(e)|-1)=10/7 \le \lambda_{\max}}$.

\begin{cor}
Let $H=(V; E,\varepsilon)$ be a hypergraph, then:

\begin{description}
\item[i)]  If $H'=(V';E', \varepsilon ')$ is the induced hypergraph on $V' \subseteq V $ and $H"=(V'; E", \varepsilon")$ obtained by removing empty hyperedges of $H'$, then
\[
 \lambda_{\min}(H) \le \lambda_{\min}(H") \le \lambda_{\max}(H") \le \lambda_{\max}(H)
 \]
\item[ii)]  If $H'=(V';E', \varepsilon ')$ is the partial hypergraph generated by $E'$: $V'= \cup_{e \in E'} \varepsilon(e), 
\varepsilon'=\varepsilon|_{E'}$) then
\[
\lambda_{\min}(H) \le \lambda_{\min}(H') \le \lambda_{\max}(H') \le \lambda_{\max}(H)
\]
\end{description}
\end{cor}
\begin{proof}\mbox{}

\begin{description}
\item[i)] In $H'$ there is no isolated vertex, therefore there is no isolated vertex in $H"$. Let $\xi= \sum u_{x'e"} x'e"$ be an eigenvector associated with
$\lambda_{max}(H")$ and $||\xi ||"=1$; we complete $\xi$ in $\eta=\sum v_{xe}xe$ with $0$ for $xe \neq x'e"$; according to the proposition \ref{firstProperties}, we have:
\[
\lambda_{max} (H") = \braket{ \mathfrak{A}"(\xi) |\xi}=\braket{ \ \mathfrak{A}(\eta)|\eta} \le \lambda_{max}(H).
\]
The reasoning is identical for $\lambda_{\min}$
\item[ii)] In $H'$, there is neither an empty hyperedge nor an isolated vertex; the proof is similar.
\end{description}
\qed
\end{proof}

\subsubsection*{Example}
 In the hypergraph shown in Figure \ref{matriceHypergarph}, if we remove $e_3$ and $x_3$ we obtain $H'$ with the adjacency matrix:

 \bordermatrix{& x_{1}e_1 & x_{2}e_1&x_{1}e_2&x_{4}e_{2}
\cr  x_{1}e_1  & 0 & 1 & +\ii& 0
 \cr x_{2}e_1 & 1 & 0 & 0 & 0 
 \cr x_{1}e_2 & -\ii & 0 & 0 & 1 
 \cr x_{4}e_{2}&  0 & 0 & 1 & 0}
 \mbox{}\\\\ 
 
 \noindent We have:  ${\displaystyle m_{E'}=4,\;  P'(x)= -(x^2-x-1)(x^2+x-1)}$ and ${\displaystyle \textrm{Spec}(\mathfrak{A}')}$ is :\\\\
${\displaystyle \lambda'_{\max} \simeq 1,618 > \lambda'_2\simeq 0,618 > \lambda'_3 \simeq -0, 618 > \lambda'_{\min} \simeq -1,618}$ confirming the interlacing. We have:  $\max\{ d(x)+|\varepsilon(e)|-2\}=2$.\\


\begin{cor}
Let $H=(V; E)$ be a $k$-uniform  hypergraph with $k\geq 2$, then:
$$
k-1 \le \lambda_{max}(H) \le \Delta(H)+ k-2 \le \Delta([H]_2)
$$
\end{cor}
\begin{proof}\mbox{}

${\displaystyle \lambda_{\max} \le  \max\{\ d(x) +|\varepsilon(e)| -2\}=  \max\{d(x)+k-2 \}= \Delta(H)+k-2 \le (k-1) \Delta(H)= \Delta([H]_2)}$, since $k \ge 2$, (it is easy to prove that if $H$ is $k$-uniform we have  $\Delta([H]_2)=(k-1) \Delta(H)$). \\
${\displaystyle \lambda_{\max} \ge r(H)-1=k-1}$.
\end{proof}


\subsection{Interlacing properties of matrix $A_H$}
Let us recall the interlacing property of Hermitian matrices: \\
Let $A$ be a Hermitian matrix $n\times n$, with ${\displaystyle \textrm{Spec}(A)=\{ \lambda_1 \ge \lambda_2 \ldots \ge \lambda_n \}}$, and $B$ a principal submatrix $t \times t$ with 
${\displaystyle \textrm{Spec}(B)= \{ \mu_1 \ge \mu_2 \ldots \ge \mu_t \}}$; then 
\[
\textrm{for}\;  s=1, \ldots , t \;:\; \lambda_s \ge \mu_s \ge \lambda_{s+n-t}.
\]

\begin{application} : Let $r=r(H)$ and $M_r$ be a principal submatrix of $A_H$ associated with $e$ with $|\varepsilon(e)|=r$: $M_r$ is an $r\times r$ matrix composed of $0$ on the diagonal and $1$ elsewhere; it is well known that $\textrm{Spec}(M_r)= \{ \mu_1=r-1, \mu_2=-1,
\ldots \mu_r=-1 \}$. By the interlacing property, we have:
\[
\lambda_{\max}=\lambda_1 \ge \mu_1=r-1 \ge \lambda_{m_E-r+1}; \;\;  \lambda_2 \ge \mu_2 =-1 \ge \lambda_{m_E-r+2} \ldots \lambda_r \ge \mu_r =-1 \ge \lambda_{m_E}= \lambda_{\min}
\]
\end{application}

\begin{application} : 
We use $A_H^{l}$;  so we need certain properties of the Hermitian matrix $n\times n-N_n$ composed of $0$ on the diagonal and $i$ above the diagonal (so $-\ii$ at the bottom of the diagonal):


\begin{lem}\label{lemPropertSpec}\textsc{}
\begin{itemize}
\item[i)] ${\displaystyle \textrm{Spec}(N_n)=\{ \nu_1 \ge \nu_2 \ldots  \ge \nu_n \}}$ is symmetric about $0$;
\item[ii)] ${\displaystyle \sum \nu_i=0}$;
\item[iii)]  ${\displaystyle \sum \nu_i^2=n(n-1)}$;
\item[iv)] ${\displaystyle  \nu_1 \ge \sqrt{n-1}}$;
\item[v)] if $n$ is odd then ${\displaystyle \nu_1 \ge \sqrt{n}}$.
\end{itemize}
\end{lem}

\begin{proof}

The arguments used in this proof are inspired by \cite{Guo2015}.
\begin{itemize}
\item[i)] The matrix $\ii N_n$ has coefficients $0$, $\pm1$ -antisymmetric, but not Hermitian-  therefore: \\
${\displaystyle \Phi(\ii N_n,x) \in \mathbb{R}[x]: \alpha \in \textrm{Spec}(\ii N_n) \Longrightarrow \overline{\alpha} \in \textrm{Spec}(\ii N_n)}$; or, if ${\displaystyle  \nu \in \textrm{Spec}(N_n)}$, \\
then ${\displaystyle  \alpha=\ii \nu \in\textrm{Spec}(\ii N_n)}$, and ${\displaystyle  \overline{\alpha}=-\ii \nu\in \textrm{Spec}(\ii  N_n)}$;\\
$\nu$ is a real number: ${\displaystyle  \overline{\alpha}=\ii \nu', \nu' \in \textrm{Spec}(N_n)}$; so ${\displaystyle  -\ii\nu=\ii\nu': \nu'=- \nu \in \textrm{Spec}(N_n)}$,
\item[ii)] $Tr(N_n)=0= \sum \nu_i$;
\item[iii)] the diagonal coefficient $c_{jj}$ of $N_n^2$ is $c_{jj}= \sum_k c_{jk} c_{kj}=\sum_k | c_{jk}|^2=n-1$, so $Tr(N_n^2)=n(n-1)=\sum \nu_i^2$;
\item[iv)] $n(n-1)=\sum \nu_i^2 \le n \nu_1^2$, therefore $\nu_1^2 \ge n-1$
\item[v)] if $n=2t+1$ then $(2t+1)2t= \sum \nu_i^2=2 \sum_{\nu_i >0} \nu_i^2 \le 2t \nu_1^2 : \nu_1^2 \ge \sqrt{2t+1}=\sqrt{n}$.
\end{itemize}

\end{proof}

Let $s=\Delta(H)$, so that the principal submatrix of $A_H^{l}$ associated with    $x\in V$ such that $d(x)=s$ is precisely $N_s$; by  Lemma \ref{lemPropertSpec}, 
$\textrm{Spec}(N_s)$ is symmetric with respect to $0$, therefore $\nu_{ \lceil \frac{s}{2} \rceil } \ge 0, \nu_{\lfloor \frac{s}{2} \rfloor} \le 0$.\\
By the  Interlacing property, we have:
\[
\lambda_{\max}=\lambda_1 \ge \nu_1; \ldots ;  \lambda_{\lceil \frac{s}{2} \rceil} \ge \nu_{\lceil \frac{s}{2} \rceil} \ge 0.
\]
These inequalities prove that $\vert \textrm{Spec}^+(A_H)\vert \ge \lceil \frac{s}{2} \rceil= {\Big\lceil \frac{\Delta(H)}{2} \Big\rceil}$.\\
In addition, we have:
\[
0 \ge \nu_{ \lfloor \frac{s}{2} \rfloor +1 } \ge \lambda_{ m_E -\lceil \frac{s}{2}\rfloor +1 } \ldots  0 \ge \nu_s \ge \lambda _{m_E}.
\]
These inequalities prove that $|\textrm{Spec}^-(A_H)| \ge \lfloor \frac{s}{2} \rfloor= \Big\lfloor \frac{\Delta(H}{2}\Big\rfloor $; moreover,
$\nu_1 \ge \sqrt{s-1}$ or $\sqrt{s}$ by  Lemma \ref{lemPropertSpec}, we have proven:

\begin{prop}\textsc{}
Let $H=(V; E)$ be a   hypergraph and $A_H$ its adjacency matriw, we have the following properties:
\begin{description}
\item[i)]  ${\displaystyle \lambda_{\max} \ge \sqrt{\Delta(H)-1}, \;\; \lambda_{\max} \ge \sqrt{\Delta(H)}}$  if ${\displaystyle\Delta(H)}$ is odd;
\item[ii)]  ${\displaystyle \vert \textrm{Spec}^+(A_H)\vert \ge  \Big\lceil \frac{\Delta(H)}{2} \Big\rceil  \textrm{and}  \;\;\;   \vert \textrm{Spec}^-(A_H)\vert \ge \Big \lfloor \frac{\Delta(H}{2} \Big\rfloor}$.
\end{description}
\end{prop}
\end{application}


\section{Characterization of hypergraphs by their adjacency matrix}
In this section we will show that the adjacency matrix fully characterizes the associated hypergraph. 
\subsection{Hypergraph Isomorphism}
Let ${\displaystyle f: H=(V;E,\varepsilon) \longrightarrow H'=(V';E', \varepsilon')}$ be an isomorphism:
\[
f: V \longrightarrow V', \; f: E\longrightarrow E';
\]
 There is therefore:
 \[
\begin{split}
&g: H=(V';E',\varepsilon') \longrightarrow H=(V; E, \varepsilon)\; \textrm{ such as}\; g: V' \longrightarrow V, g: E'\longrightarrow E\; \\
&\textrm{with} \; f\circ g= Id\; \textrm{et}\; g \circ f=Id;
\end{split}
\]
thus, we have $f(\varepsilon(e))=\varepsilon'(f(e))$.\\

\noindent This isomorphism $f$ naturally translates into\\
 ${\displaystyle \tilde{f}:VE \longrightarrow V'E'}$ by ${\displaystyle \tilde{f}(xe)=f(x)f(e)}$ because
${\displaystyle x \in \varepsilon(e) \Longrightarrow f(x) \in f(\varepsilon(e))=\varepsilon'(f(e));\; \tilde{f}}$ being bijective, this extends to a 
${\displaystyle \mathbb{C}-\textrm{isomorphism}\; \tilde{f} \in \mathcal{L}(<VE>, <V'E'>)}$. Therefore:
\begin{description}
\item[-] we have ${\displaystyle \mathfrak{A}_H(xe)= \sum_{y \neq x} ye + \sum_{a \neq e} \pm \ii xa \in <VE>}$; so,  ${\displaystyle \mathfrak{A}_H \in \textrm{End}(<VE>)}$;
\item[-]  The following diagram is commutative:
\[ 
\begin{tikzcd}[sep=4cm]
<VE>  \arrow{r}{ \mathfrak{A}_H} \arrow[swap]{d}{\tilde{f}} & <VE> \arrow{d}{\tilde{f}} \\%
<V'E'> \arrow{r}{ \mathfrak{A}_{H'}}&  <V'E'>
\end{tikzcd}
\]\\
\noindent Indeed: Let ${\displaystyle  \tilde{f}(xe)=x'e'}$; let us note ${\displaystyle \beta = \tilde{f} \circ \mathfrak{A}_H \circ \tilde{f}^{-1}:}$\\
 ${\displaystyle  \beta(x'e')=\tilde{f} \circ \mathfrak{A}_H\circ \tilde{f}^{-1}(x'e')=\beta(x'e')=\tilde{f} \circ \mathfrak{A}_H(xe)}$, \\
where ${\displaystyle f(x)=x', f(e)=e'}$;\\
\[
\begin{split}
 \beta(x'e')&= \tilde{f}( \sum_{y\neq x} ye +\sum_{a\neq e} \pm \ii xa)=\sum_{f(y)\neq f(x)} f(y)f(e)+\sum_{f(a)\neq f(e)} \pm \ii f(x)f(a)= \\
 &\sum_{y' \neq x'} y'e' +\sum_{a'\neq e'} \pm \ii x'a'=\mathfrak{A}_{H'}(x'e').
 \end{split}
 \]
 We can therefore conclude that:
\[
\mathfrak{A}_{H'}= \tilde{f} \circ \mathfrak{A}_H \circ \tilde{f}^{-1}
\]

\end{description}

Conversely: Let ${\displaystyle \tilde{f}  : <VE> \longrightarrow <V'E'>}$ be a linear isomorphism such that: ${\displaystyle \tilde{f} \vert_{VE} : VE \longrightarrow V'E'}$ is bijective (so if $V'=V, E'=E$, then it is a change of basis), and satisfying:

\begin{equation}\label{isomHyper}
\mathfrak{A}_{H'}= \tilde{f} \circ \mathfrak{A}_H \circ \tilde{f}^{-1}\Longleftrightarrow \mathfrak{A}_{H'}\circ  \tilde{f} = \tilde{f} \circ \mathfrak{A}_{H}
\end{equation}
we want to prove that $H \simeq H'$ :\\

\begin{description}
\item[a)] let $\xi=xe\in VE$, $\tilde{f} $ being a bijection there is a unique $x'e'\in V'E'$ such that:  $\tilde{f}(xe)=x'e'$; we have:
\begin{equation}\label{equIso1}
 \mathfrak{A}_{H'} \circ\tilde{f}(xe) = \mathfrak{A}_{H'} (x'e')=\sum_{y' \neq x'}y'e'+ \sum_{a'\neq e'} \pm \ii x'a';
\end{equation}
 in the same way:
 \begin{equation}\label{equIso2}
 \tilde{f} \circ\mathfrak{A}_{H} (xe) = \tilde{f}\bigl(\sum_{y \neq x}ye+ \sum_{a\neq e} \pm \ii xa\bigr)= \sum_{y \neq x}\tilde{f}(ye)+ \sum_{a\neq e} \pm \ii \tilde{f}(xa). 
 \end{equation}
 
 Since $\tilde{f}(xe)=x'e'$,  and by  identification, between equation \ref{equIso1} and equation \ref{equIso2}, we obtain:
 \[
  \begin{split}
 & \textrm{for all}\; y\in V,\; \textrm{such that }\; y\in \varepsilon(e): \tilde{f}(ye)=y'e',\; \textrm{for all}\; y'\in \varepsilon(e'), \textrm{and}\\
 &\textrm{for all}\; a\in E,  \textrm{such that }\;  \varepsilon(a)\ni x,  \tilde{f}(xa)=x'a', \textrm{for all}\;   a'\in E' , \textrm{such that }\; \varepsilon(a')\ni x'.
   \end{split}
\]
Therefore, we can define  a function $\phi: E\longrightarrow E'$ such that $\phi(e)=e'$ verifying:
\[
\textrm{for all} \;\; x \in \varepsilon(e):\; f(xe)=x' \varphi(e), \forall \; e\in E;
\]
and another function denoted in the same way $\phi: V\longrightarrow V'$ such that  that $\varphi(x)=x'$ verifying:
$$
\forall \; \varepsilon(e) \ni x:\; f(xe)=\varphi(x)e',  \forall \; x\in V.
$$
It is easy to verify that these functions  are indeed functions.

\item[b)] Functions $\varphi$ is surjective:\\
Let $a' \in E'$; let us choose ${\displaystyle x' \in \varepsilon'(a') : x'a' \in V'E'}$, since $\tilde{f}$ is bijective and therefore surjective, there therefore exists $xa \in VE$ such that $\tilde{f}(xa)=x'a'$, but $\tilde{f}(xa)=x"\varphi(a)$, by definition, therefore $x'=x"$ and $a'=\varphi(a)$. Consequently, we have $\vert E\vert \ge \vert E'\vert$.\\
Now,  let $x' \in V'$; let us choose ${\displaystyle a' : \varepsilon'(a')\ni x',  x'a' \in V'E'}$, since $\tilde{f}$ is bijective and therefore surjective, there therefore exists $xa \in VE$ such that $\tilde{f}(xa)=x'a'$, but $\tilde{f}(xa)=\varphi(x)$e", by definition, therefore $e'=e"$ and $x'=\varphi(x)$. Consequently, we have $\vert V\vert \ge \vert V'\vert$.

\item[c)]  Function $\varphi$  injective:\\
Let $xe, ya\in VE$, $\tilde{f}$ being a bijection, we have:
\[
\tilde{f}(xe)= \tilde{f}(ya)= x'e'= y'a'= x'\varphi(e)= y'\varphi(a);
\]
hence, $\varphi(e)= \varphi(a)$ and $x'=y'$, so
$\varphi(e)=e'= \varphi(a)=a'$; so $\vert E\vert\leq \vert E'\vert$.\\
In the same way, let $xe, ya\in VE$, $\tilde{f}$ being a bijection, we have:
\[
\tilde{f}(xe)= \tilde{f}(ya)= x'e'= y'a'= \varphi(x)e'= \varphi(y)a';
\]
hence, $\varphi(x)= \varphi(y)$ and $e'=a'$, so
$\varphi(x)=x'= \varphi(y)=y'$; so $\vert V\vert\leq \vert V'\vert$.\\
\item[d)] From points $b)$ and $c)$, we can conclude that: $\tilde{f}(xe)=\varphi(x)\varphi(e)$, for all $xe\in VE$. \\

Therefore, these two bijections
\[
\varphi: V \longrightarrow V',\; \varphi: E\longrightarrow E';
\]
define a bijection:\\
${\displaystyle \varphi: H \longrightarrow H'}$ which is an isomorphism, in fact:\\
 ${\displaystyle x \in \varepsilon(e) \Longleftrightarrow xe \in VE \Longleftrightarrow f(xe)=\varphi(x)\varphi(e)\Longleftrightarrow\varphi(x) \in \varepsilon'(\varphi(e))}$.
\end{description}
We have just proven the
\begin{thm}\label{characteristionHypergraph}
Let $H=(V;E,\varepsilon)$ and $H'=(V';E', \varepsilon')$ be two hypergraphs; The following two properties are equivalent: 
\begin{description}
\item[i)] Hypergraphs $H$ and $H'$ are isomorphic: $H \simeq H'$;
\item[ii)] there exists a linear isomorphism ${\displaystyle \tilde{f} \in Iso\left(<VE>, <V'E'>\right)}$ such that: ${\displaystyle \tilde{f} \vert_{VE}: VE \longrightarrow V'E'}$ is bijective such that: ${\displaystyle  \mathfrak{A}_{H'}= \tilde{f} \circ \mathfrak{A}_H \circ \tilde{f}^{-1}}$.
\end{description}
\end{thm}

\section{Graph associated with a hypergraph}\label{associateGraph}
Using the hypergraph adjacency matrix, we can define a graph that we will call \emph{associated graph} \index{Associated graph} to $H=(V;E, \varepsilon)$ is the simple graph $\Gamma_H=(VE;U)$ defined by:

\begin{equation}
\{xe,x'e'\} \in U \; \textrm{if and only if}\; a_{xe,x'e'} \neq 0.
\end{equation}
 
As illustration Figure \ref{assoGraph} gives the $\Gamma_H$ of the hypergraph of figure 7.1:
\begin{figure}[H]
\psfrag{a}{$21$}
\psfrag{b}{$11$}
\psfrag{c}{$12$}
\psfrag{d}{$42$}
\psfrag{e}{$43$}
\psfrag{f}{$33$}
\psfrag{g}{$31$}
\
\centering
\includegraphics[width=5cm]{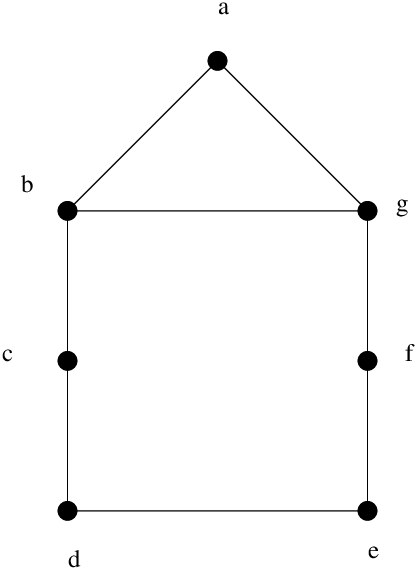}
\caption{This figure shows  the associatd graph of   hypergraph of Figure \ref{matriceHypergarph}. In this figure we denote: $x_{i}e_{j}$ by $ij$.}
\label{assoGraph}
\end{figure}
The adjacencey matrix of the associated graph of Figure \ref{assoGraph}  is given above:

\bordermatrix{&x_1e_1 &x_2e_1&x_3e_1&x_1e_2&x_4e_2&x_3e_3&x_4e_3
\cr  x_1e_1  & 0 & 1& 1& 1& 0 & 0 & 0   
 \cr x_2e_1 & 1& 0 & 1 & 0 & 0 & 0 & 0 
 \cr x_3e_1 & 1& 1 & 0 & 0& 0 &1 & 0
 \cr x_1e_2& 1& 0 & 1& 0 & 1 & 0 & 0
 \cr x_4e_2 & 0 & 0 & 0 & 1& 0 & 0 & 1
 \cr x_3e_3 & 0 & 0& 1 & 0 & 0 & 0 & 1
 \cr x_4e_3& 0 & 0 & 0 & 0 & 1& 1 & 0}
\mbox{}\\ 
 
\noindent This matrix will be denoted: $A(\Gamma_{H})$.\\

\noindent As illustration, Figure \ref{assoGraph} gives the associated graph $\Gamma_H$ of the hypergraph of Figure \ref{matriceHypergarph}.

\begin{prop}
Let  $H=(V;E,\varepsilon)$  be a  simple linear hypergraph with $\delta(H)\geq 2$; let   $H^{*}= (V^*; E^*, \varepsilon^{*})$ be its dual , then
\[
\Gamma_H\simeq \Gamma_{H^{*}}
\]

\end{prop}

\begin{proof}\mbox{}\\
Note that under the conditions of the Proposition, the hypergraph $H^{*}= (V^*; E^*, \varepsilon^{*})$ is simple linear with $\delta(H^{*})\geq 2$.\\
Since  ${\displaystyle V^*\ast E^*:= \{(X_{e}, E_{x}):X_{e}\in \varepsilon^{*}(E_{x}), \; \textrm{iff}\; \varepsilon(e)\ni x\}}$, let's define the function:
 \[
 \begin{array}{ccccc}
h & : & \Gamma_H  & \to & \Gamma_{H^{*}} \\
 & & (x,e) & \mapsto & (X_{e}, E_{x}) \\
\end{array}
\]
 \[
 \begin{split}
  &h((x,e))= h((y,a))= (X_{e}, E_{x})=(X_{a}, E_{y})\Longrightarrow  X_{e}=X_{a}\wedge E_{x}=E_{y}\\
  &\Longrightarrow e=a \wedge H(x)=H(y) \Longrightarrow x=y,
  \end{split}
  \]
 since $H$ is simple and linear.  Moreover: ${\displaystyle \vert X\ast E\vert =  \vert X^{*}\ast E^{*}\vert }$; so, $h$ is a bijection.\\
 Let  ${\displaystyle \{(x, e); (y, a)\}\in E(\Gamma_H)\}}$, we have two cases:
 \begin{itemize}
\item[-] ${\displaystyle x=y\Longleftrightarrow H(x)=H(y) \Longleftrightarrow E_{x} =E_{y})\Longleftrightarrow \{((X_{e}, E_{x}); (X_{a}, E_{y}) \}\in E(\Gamma_{H^{*}})}$
\item[-] ${\displaystyle e=a \Longleftrightarrow  X_{e}=X_{a}\Longleftrightarrow \{((X_{e}, E_{x}); (X_{a}, E_{y}) \}\in E(\Gamma_{H^{*}})}$, 
by applying the same type of reasoning
\end{itemize}

\end{proof}

\begin{figure}[H]
\psfrag{a}{$21$}
\psfrag{b}{$11$}
\psfrag{c}{$12$}
\psfrag{d}{$42$}
\psfrag{e}{$43$}
\psfrag{f}{$33$}
\psfrag{g}{$31$}
\centering
\includegraphics[width=5cm]{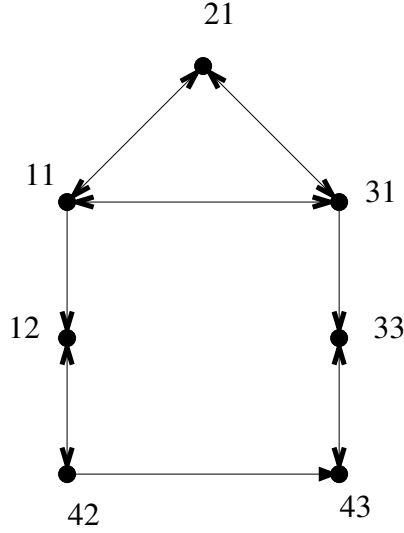}
\caption{This figure shows  the mixed graph of   hypergraph of Figure \ref{matriceHypergarph}. In this figure we denote: $x_{i}e_{j}$ by $ij$.}
\label{mixGraph}
\end{figure}
\begin{rem} 
The matrix $A_H$ can also be interpreted as adjacency matrix of a "mixed digraph" $G_{H}$ shown Figure \ref{mixGraph}. Clearly the undergraph of $G_{H}$ is $\Gamma_H$. The graph $G_{H}$  will be called \emph{associated mixed graph} \index{associated mixed graph}.
\end{rem}
We now give a Corollary to the Theorem \ref{characteristionHypergraph}:

\begin{cor}
Let $H=(V;E,\varepsilon)$ and $H'=(V';E', \varepsilon')$ be two hypergraphs; The following two properties are equivalent: 
\begin{description}
\item[i)] Hypergraphs $H$ and $H'$ are isomorphic: $H \simeq H'$;
\item[ii)] Associated mixed graphs $G_{H}$ and $G_{H'}$ are isomorphic: $G_{H}\simeq G_{H'}$.
\end{description}

\end{cor}
\begin{proof}

Directly by applying  Theorem \ref{characteristionHypergraph}.
\end{proof}

\subsection{Properties of the adjacency matrix of the associated graph}
We can define several equivalence relations on the vertices of $\Gamma_{H}$:
\begin{itemize}
\item[-] let $xe, x'e'\in VE$: $xe\mathcal{R} x'e' \iff e=e'$;
\item[] another can be defined as follows:
\item[-] let $xe, x'e'\in VE$: $xe\mathcal{R}' x'e' \iff x=x'$.
\end{itemize}

\subsubsection{Properties of $\mathcal{R}$}
The    \emph{line-graph} of $H$ is the graph
$L(H) = (V'; E')$ such that:

\begin{enumerate}
\item $V' := E$;
\item $E': \{e,e' \}\in E'$ ($e \neq e'$) if and only if $\varepsilon(e)\cap \varepsilon(e)' \neq \emptyset$.
\end{enumerate}

We will denote by: $\mathcal{R}(xe)$ the equivalence class of $xe$.\\
This equivalence relation allows us to define a simple graph structure on the quotient $VE/\mathcal{R}$ as follows: ${\displaystyle \Gamma_H/ \mathcal{R} = (VE/\mathcal{R} ;U_{\mathcal{R}})}$ with:

\[
\{ \mathcal{R}(xe), \mathcal{R}(x'e') \} \in U_{\mathcal{R}} \; \text{if and only if} \; \varepsilon(e) \cap \varepsilon(e') \neq \emptyset.
\]
(This relation is well-defined because it is independent of $x,x',e,e'$).\\

Therefore, we have the canonical map: 
\[
 \pi_{\mathcal{R}} : \Gamma_H=(VE;U) \longrightarrow \Gamma_H/ \alpha =(VE/\mathcal{R}; U_{\mathcal{R}})
 \]
  which is a (surjective) morphism of simple graphs in the sense that if $xe$ is adjacent to $x'e'$, then $\pi_{\mathcal{R}} (xe)$ and $\pi_{\mathcal{R}}(x'e')$ are equal or adjacent: \\
indeed, 
\begin{itemize}
\item[-] if $e=e'$  we have $\mathcal{R}(xe)=\mathcal{R}(x'e') \}$ , by definition;
\item[-] if $e\neq e'$, we have ${\displaystyle \{ \mathcal{R}(xe), \mathcal{R}(x'e') \} \in U_{\mathcal{R}}}$ if and only if  $x= x' \in \varepsilon(e) \cap \varepsilon(e')$.
\end{itemize}

\begin{prop} 
Let $\Gamma_H$ be the associated graph $H$; then, the quotient graph $\Gamma_H/ \mathcal{R} =(VE/\mathcal{R}; U_{\mathcal{R}})$ 
is isomorphic to $L(H)=(V';E')$ the line-graph of $H$.
\end{prop}

\begin{proof}

Let us define $f: E \longrightarrow VE\textfractionsolidus\mathcal{R}$ as follows:
\[
f(e)=\mathcal{R}(xe)\; \textrm{for one}\; x\in \varepsilon(e),
\] 
(this is clearly independent of $x$); then $f$is clearly surjective  and

\[
f(e)=f(e') \Longrightarrow \mathcal{R}(xe)= \mathcal{R}(x'e') \Longrightarrow e=e'; \; \textrm{so} \; f\; \textrm{is a bijection}.
\]
Recall that $L(H)=(V';E')$ with $V'=E$ and $\{ e, e'\} \in E'$ iff ${\displaystyle \varepsilon(e) \cap \varepsilon(e') \neq \emptyset}$.\\
Hence ${\displaystyle f: V'=E \longrightarrow VE/\mathcal{R}}$ is bijective, moreover it is a graph isomorphism:
\begin{description}
\item[-] if $ \{ e,e'\} \in E'$,  then $\varepsilon(e) \cap \varepsilon(e') \neq \emptyset$: choose $x$ in this intersection; $f(e)=\mathcal{R}(xe), f(e')=
\mathcal{R}(xe')$ and $\{ \mathcal{R}(xe), \mathcal{R}(xe') \} \in U_{\mathcal{R}}$ since $\varepsilon(e) \cap \varepsilon(e') \neq \emptyset$
\item[-] if $f(e)=\mathcal{R}(xe)$ is adjacent to $f(e')=\mathcal{R}(x'e')$, then $\{ \mathcal{R}(xe), \mathcal{R}(xe') \} \in U_{\mathcal{R}}$ ; therefore $\varepsilon(e) \cap \varepsilon(e') \neq \emptyset$ and $\{e,e' \} \in E'$.
\end{description}
\end{proof}

\subsubsection{Properties of $\mathcal{R}'$}
Let $H = (V; E, \varepsilon)$ be a hypergraph, the \emph{2-section} of  $H$ is the graph, denoted by $[H]_{2}$, where:
\begin{itemize}
\item the vertices set  of $[H]_{2}$ is the  set $V$ of vertices of $H$;
\item   two distinct vertices $x,y$ form an edge in $[H]_{2}$ if and only if  there is $e\in E$  such that $x,  y\in \varepsilon(e)$.
\end{itemize}

As with the first equivalence relation, $\mathcal{R}'$ allows us to define a simple graph structure on the quotient $VE/\mathcal{R}'$ as follows: ${\displaystyle \Gamma_H/ \mathcal{R}' = (VE/\mathcal{R}' ;U_{\mathcal{R}'})}$ with:

\[
\{\mathcal{R}'(xe), \mathcal{R}'(x'e') \} \in U_{\mathcal{R}'} \; \text{if and only if there exists} \; a \in E \; \text{such that}\;\; x,x' \in \varepsilon(a)
\]
As above, the relation is well-defined.
Therefore, we also have a canonical map:

\[
\pi_{\mathcal{R}'} : \Gamma_H=(VE;U) \longrightarrow \Gamma_H/ \mathcal{R}' =(VE/\mathcal{R}'; U_{\mathcal{R}'})
\]
which is a (surjective) morphism of simple graphs in the sense that if $xe$ is adjacent to $x'e'$, then $\pi_{\mathcal{R}'} (xe)$ and $\pi_{\mathcal{R}'}(x'e')$ are equal or adjacent: indeed,
\begin{itemize}
\item[-] if $x=x'$, then $\mathcal{R}(xe)=\mathcal{R}(x'e') \}$, by definition;
\item[-] If $x\neq x'$, we have ${\displaystyle \{ \mathcal{R}'(xe), \mathcal{R}'(x'e') \} \in U_{\mathcal{R}}}$ if and only if $x\neq x' \in \varepsilon(e) =\varepsilon(e')$.
\end{itemize}


\begin{prop} 
Let $\Gamma_H$ be the associated graph $H$; the quotient graph $\Gamma_H/ \kappa =(VE/\mathcal{R}'; U_{\mathcal{R}'})$ is isomorphic to $[H]_2=(V;D) $ the 2-section of $H$.
\end{prop} 

\begin{proof}

Recall that $\{x,y\} \in D $ iff exists $e \in E $ such that $x,y \in \varepsilon(e)$.\\
Define $g: V \longrightarrow VE/\mathcal{R}'$ by the rule 
\[
g(x)=\mathcal{R}'(xe)\;  \textrm{where}\;  e\; \textrm{ is choose such that}\; x\in \varepsilon(e);
\]
this is  independent of $e$ since if $x \in \varepsilon(e')$ then $\mathcal{R}'(xe)=\mathcal{R}'(xe')$).\\
Clearly  $g$ is surjective and $g(x)=g(y) \Longrightarrow \mathcal{R}'(xe)= \mathcal{R}'(x'e') \Longrightarrow x=y$; so, $g$ is a bijection.\\
This bijection gives rise to an isomorphism of graphs:

\begin{description}
\item[-] if $\{x,y\} \in D: x,y \in \varepsilon(e) \Longrightarrow \{ \mathcal{R}'(xe), \mathcal{R}'(ye) \} \in U_{\mathcal{R}'}$;
\item[-] if $\mathcal{R}'(xe)$ is adjacent to $\mathcal{R}'(x'e')$, there exists $a \in E$ such that $x,y \in \varepsilon(a): \{x,y \} \in D$.
\end{description}
\end{proof}

The adjacency matrix $B_H$ of $\Gamma_H$ verify $b_{xe,x'e'} =|a_{xe,x'e'}|$ and may be decomposed as $B_H=B_H' +B_H"$ where
$b'_{xe,x'e'}= Re ( a_{xe,x'e'})$ and $b"_{xe,x'e'}=|Im(a_{xe,x'e'}|$.\\
So for every $x$ in $V$ : $d_h(x)= \sum_{x'e'}|Im(a_{xe,x'e'}|+1$ (independent  of $e$ such that $x \in \varepsilon(e)$), \\
and for every $e$ in $E$ :$|\varepsilon(e)|= \sum_{x'e'} Re(a_{xe,x'e'})+1$ (independent of $x \in \varepsilon(e)$).\\
We have:


\begin{prop}\mbox{}
\begin{description}
\item[i)] If $H$ is $k$-regular then $\lambda=k-1 \in Sp(B"_H):   B"_H(1)= (k-1) 1$;
\item[ii)] if $H$ is $k$-uniform then $\lambda=k-1 \in Sp(B'_H):  B'_H(1)= (k-1) 1$.
\end{description}
\end{prop}
\begin{proof}
Clear
\end{proof}


\subsection{Laplacian}
Let $D=(d_{xe,xe})$ be  the diagonal matrix where $d_{xe,xe}= d(x)+|\varepsilon(e)|-2= \sum_{x'e'} |a_{xe,x'e'}|$ and $\mathcal{D}$ the associated endomorphism. We define $\mathcal{L}=\mathcal{D} - \mathfrak{A}$ as \index{laplacian of $H$} \emph{laplacian} of $H$, and we note:
\[
\textrm{Spec}(\mathcal{L}) =\{ \nu_1 \ge \nu_2 \ldots , \nu_{m_E} \}\; \textrm{the spectrum of}\;  \mathcal{L}.
\]

\begin{prop}
Let $H=(V; E, \varepsilon)$ be a hypergraph and let $\mathcal{L}$ be the associated Laplacian, then:
\begin{itemize}
\item[i)] $\mathcal{L}$ is a semi-positive definite endomorphism;
\item[ii)] $\sum_{1 \le i \le m_E} \nu_i = \sum_{xe} d(x)+|\varepsilon(e)|-2$.
\end{itemize}
\end{prop}

\begin{proof}
\mbox{}

\begin{itemize}
\item[i)] Let $L=(l_{xe,x'e'})$ be the matrix of $\mathcal{L}$ in the base $VE$. $L$ being diagonalizable, it suffices to prove that the $\nu_i$ are non-negative: let ${\displaystyle \xi=\sum_{xe\in VE} \alpha_{xe}xe}$ be an eigenvector associated with $\nu$, let us choose $ya$ such that ${\displaystyle \alpha_{ya}=max\{ \vert\alpha_{xe}\vert: \alpha_{xe}\in \xi|}$, we have:
\[
\nu \alpha_{ya}= \sum_{xe\in VE}l_{ya, xe}\alpha_{xe}= l_{ya, ya}+ \sum_{\substack{ xe\in VE\\ xe\neq ya}}l_{ya, xe}\alpha_{xe};
\]
therefore
\[
\nu =l_{ya, ya}+ \sum_{\substack{ xe\in VE\\ xe\neq ya}}l_{ya, xe}\frac{\alpha_{xe}}{\alpha_{ya}};
\]
 thus,
\[
\begin{split}
\vert\sum_{\substack{ xe\in VE\\ xe\neq ya}}l_{ya, xe}\frac{\alpha_{xe}}{\alpha_{ya}}\vert\leq \sum_{\substack{ xe\in VE\\ xe\neq ya}}\vert l_{ya, xe}\vert &= \sum_{xe\in VE}\vert a_{ya, xe}\vert= d(y)+\varepsilon(a)-2\\
&= d_{ya, ya}= l_{ya, ya}
\end{split}
\]
we conclude that: $\nu\geq 0$.
\begin{rem}
This is a general fact: $L$ is a diagonal dominant Hermitian matrix; more precisely, if $A=(a_{i,j})$ is Hermitian satisfying $a_{i,i} \ge 0$ and $\forall i \;\; |\sum_j a_{i,j} | \le a_{i,i}$, then $A$ is semi-positive definite.
\end{rem}
\item[ii)] $Tr(\mathcal{L})= \sum \nu_i =Tr(L)=Tr(D)= \sum_{xe} d(x)+ |\varepsilon(e)|-2$.
\end{itemize}

\end{proof}
\begin{rem}
Another interesting Laplacian matrix is the one we can express starting from the associated graph, (see Paragraph \ref{associateGraph}):
\[
L(\Gamma_{H}) = D - A(\Gamma_{H}).
\]
\end{rem}
\section{Conclusion}
This article introduced a new type of matrix associated with a hypergraph; this matrix is based on a new definition (representation) of hypergraphs. 
Naturally, a more detailed study should be developed. We can also consider developing an algorithm that reconstructs the hypergraph from its adjacency matrix. 
This construction gives rise to two graphs: the associated graph and the associated mixed graph. An interesting question arises:\\

\noindent "What mixed graphs generate an adjacency matrix of a hypergraph?"\\

\noindent This question can also be asked for associated graphs.
Furthermore, other problems can be studied, for example:
does the adjacency matrix provide information about the colorings of the hypergraph?

\nocite{*}
 \bibliographystyle{plain}

 \bibliography{biblio}

 \end{document}